\documentclass[12pt]{article}
\usepackage{amsmath, amsthm,amssymb}
\usepackage[english]{babel}
\usepackage[latin1]{inputenc}
\usepackage{array}
\newtheorem{theorem}{Theorem}
\newtheorem{corollary}[theorem]{Corollary}
\newtheorem{definition}[theorem]{Definition}

\begin{document}

\title{On Convolved Generalized Fibonacci and Lucas Polynomials}

\author{José L. Ramírez$^{1}$\\
$^1$Instituto de Matemáticas y sus Aplicaciones.  \\ 110221, Universidad Sergio Arboleda,\\   Colombia\\
josel.ramirez@ima.usergioarboleda.edu.co\\[2pt]
}

\maketitle

\begin{abstract}
We define the convolved $h(x)$-Fibonacci polynomials as an extension of the classical convolved Fibonacci numbers.  Then we give some combinatorial formulas involving the $h(x)$-Fibonacci and $h(x)$-Lucas polynomials. Moreover we obtain the convolved  $h(x)$-Fibonacci polynomials from a family of  Hessenberg matrices.

{\bf AMS Subject Classification:} 11B39, 11B83.

{\bf Key Words and Phrases:} Convolved $h(x)$-Fibonacci polynomials, $h(x)$-Fibonacci polynomials, $h(x)$-Lucas polynomials, Hessenberg matrices.
\end{abstract}

\section{Introduction}
Fibonacci numbers and their generalizations have many
interesting properties and applications to almost every fields of science and art  (e.g., see \cite{koshy}). The Fibonacci numbers $F_n$ are the terms of the sequence ${0, 1, 1, 2, 3, 5, . . .}$, wherein each term is the sum of the two previous terms, beginning with the values $F_0 = 0$ and $F_1 = 1$.\\
Besides the usual Fibonacci numbers many kinds of generalizations of these
numbers have been presented in the literature. In particular, a generalization is the $k$-Fibonacci Numbers. \\
For any positive real number $k$, the $k$-Fibonacci sequence, say $\{F_{k,n}\}_{n\in \mathbb{N}}$, is defined recurrently by
\begin{align}
F_{k,0}=0, \  \ F_{k,1}=1,  \   \  F_{k,n+1}=kF_{k,n}+F_{k,n-1},  \ n\geqslant 1 \label{eq1}
\end{align}

In \cite{falcon1},  $k$-Fibonacci numbers  were found by studying the recursive application of two geometrical transformations used in the  four-triangle longest-edge (4TLE) partition.   These numbers have been studied in several papers; see  \cite{ falcon1, CEN,  falcon4, falcon5,  falcon2,    Asalas}.\\

The convolved Fibonacci numbers $F_j^{(r)}$ are defined by $$(1-x-x^2)^{-r}=\sum_{j=0}^{\infty}F_{j+1}^{(r)}x^j,  \ \ \ r\in \mathbb{Z}^+.$$  If $r=1$ we have classical Fibonacci numbers.These numbers have been studied in several papers; see  \cite{ com2,  com4, com3}. Convolved $k$-Fibonacci numbers have been studied in \cite{RAM}.

Large classes of polynomials can be defined by Fibonacci-like recurrence relation and yield Fibonacci numbers \cite{koshy}.  Such polynomials, called Fibonacci polynomials,
were studied in 1883 by the Belgian mathematician Eugene Charles Catalan and the German mathematician E. Jacobsthal.
The polynomials $F_n(x)$ studied by Catalan are defined by the recurrence relation
\begin{align}
F_{0}(x)=0, \  \ F_{1}(x)=1,  \   \  F_{n+1}(x)=xF_{n}(x)+F_{n-1}(x),  \ n\geqslant 1. \label{eq11}
\end{align}
The Fibonacci polynomials studied by Jacobsthal are defined by
\begin{align}
J_{0}(x)=1, \  \ J_{1}(x)=1,  \   \  J_{n+1}(x)=J_{n}(x)+xJ_{n-1}(x),  \ n\geqslant 1. \label{eq12}
\end{align}
The Lucas polynomials $L_n(x)$, originally studied in 1970 by Bicknell, are defined by
\begin{align}
L_{0}(x)=2, \  \ L_{1}(x)=x,  \   \  L_{n+1}(x)=xL_{n}(x)+L_{n-1}(x),  \ n\geqslant 1. \label{eq13}
\end{align}
In \cite{ayse}, the authors introduced the $h(x)$-Fibonacci polynomials. That generalize 
Catalan's Fibonacci polynomials $F_n(x)$ and the $k$-Fibonacci numbers $F_{k,n}$.   In this paper, we introduce the convolved $h(x)$-Fibonacci polynomials and we obtain new identities.

\section{Some Properties of $h(x)$-Fibonacci Polynomials and $h(x)$-Lucas Polynomials}

\begin{definition}
Let $h(x)$ be a polynomial with real coefficients. The $h(x)$-Fibonacci polynomials $\{F_{h,n}(x)\}_{n\in \mathbb{N}}$ are defined by the recurrence relation
\begin{align}
F_{h,0}(x)=0, \  \ F_{h,1}(x)=1,  \   \  F_{h,n+1}(x)=h(x)F_{h,n}(x)+F_{h,n-1}(x),  \ n\geqslant 1. \label{eq22}
\end{align}
\end{definition}

For $h(x)=x$ we obtain Catalan's Fibonacci polynomials, and for $h(x)=k$ we obtain $k$-Fibonacci numbers.  For $k=1$ and $k=2$ we obtain the usual Fibonacci numbers and the Pell
numbers.

The characteristic equation associated with the recurrence relation (\ref{eq22}) is $v^2=h(x)v+1$.  The roots of  this equation are
\begin{align*}
r_1(x)=\frac{h(x)+\sqrt{h(x)^2+4}}{2},    \hspace{1cm}   r_2(x)=\frac{h(x)-\sqrt{h(x)^2+4}}{2}.
\end{align*}
Then we have the following basic identities:
\begin{align} \label{iden}
r_1(x)+r_2(x)=h(x),\  \ \hspace{0.5cm}  r_1(x)-r_2(x)=\sqrt{h(x)^2+4},\  \ \hspace{0.5cm} \  \  r_1(x)r_2(x)=-1.
\end{align}
Some of the properties that the $h(x)$-Fibonacci polynomials verify are summarized bellow (see \cite{ayse} for the proofs).

\begin{itemize}
\item Binet formula: $F_{h,n}(x)=\frac{r_1(x)^n-r_2(x)^n}{r_1(x)-r_2(x)}$.
\item Combinatorial formula: $F_{h,n}(x)=\sum_{i=0}^{\lfloor (n-1)/2 \rfloor}\binom{n-1-i}{i}h^{n-1-2i}(x)$.
\item Generating function: $g_f(t)=\frac{t}{1-h(x)t-t^2}$.
\end{itemize}

\begin{definition}
Let $h(x)$ be a polynomial with real coefficients. The $h(x)$-Lucas polynomials $\{L_{h,n}(x)\}_{n\in \mathbb{N}}$ are defined by the recurrence relation
\begin{align}
L_{h,0}(x)=2, \  \ L_{h,1}(x)=h(x),  \   \  L_{h,n+1}(x)=h(x)L_{h,n}(x)+L_{h,n-1}(x),  \ n\geqslant 1. \label{eq33}
\end{align}
\end{definition}
For  $h(x)=x$ we obtain the Lucas polynomials, and for $h(x)=k$ we have the $k$-Lucas  numbers. For $k=1$ we obtain the usual  Lucas  numbers. 

Some properties that the $h(x)$-Lucas numbers verify are summarized bellow (see \cite{ayse} for the proofs).
\begin{itemize}
\item Binet formula: $L_{h,n}(x)=r_1(x)^n+r_2(x)^n$.
\item Relation with $h(x)$-Fibonacci polynomials: $L_{h,n}(x)=F_{h,n-1}(x)+F_{h,n+1}(x), \ n \geqslant 1$.
\end{itemize}

\section{Convolved $h(x)$-Fibonacci Polynomials}
\begin{definition}
The convolved $h(x)$-Fibonacci polynomials $F_{h,j}^{(r)}(x)$ are defined by $$g_h^{(r)}(t)=(1-h(x)t-t^2)^{-r}=\sum_{j=0}^{\infty}F_{h,j+1}^{(r)}(x)t^j,  \ \ r\in \mathbb{Z}^+.$$ \end{definition}
 Note that
\begin{align}
F_{h, m+1}^{(r)}(x)=\sum_{j_1+j_2+\cdots +j_r=m}F_{h, j_1+1}(x)F_{h, j_2+1}(x)\cdots F_{h, j_r+1}(x).
\end{align}
Moreover, using a result of Gould \cite[p. 699]{GOU} on Humbert polynomials (with $n = j, m = 2,
x = h(x)/2, y = -1, p = -r$ and $C = 1$), we have
\begin{align}\label{hum}
F_{h,j+1}^{(r)}(x)=\sum_{l=0}^{\lfloor j/2 \rfloor}\binom{j+r-l-1}{j-l}\binom{j-l}{l}h(x)^{j-2l}.  
\end{align}

If $r=1$ we obtain the combinatorial formula of $h(x)$-Fibonacci polynomials. In Table\ref{tabla1} some polynomials of convolved $h(x)$-Fibonacci polynomials are provided. The purpose of this paper is to investigate the properties of these polynomials.

\begin{table}[h]
\centering
\begin{tabular}{|>{$}c<{$}|>{$}c<{$}| >{$}c<{$}|>{$}c<{$}|}\hline
n & F_{h, n}^{(1)}(x) & F_{h, n}^{(2)}(x) &  F_{h, n}^{(3)}(x)  \\ \hline
0 & 0 & 0 &0\\
1& 1 & 1 & 1 \\
 2&h & 2 h & 3 h \\
 3&h^2+1 & 3 h^2+2 & 6 h^2+3 \\
 4&h^3+2 h & 4 h^3+6 h & 10 h^3+12 h \\
 5&h^4+3 h^2+1 & 5 h^4+12 h^2+3 & 15 h^4+30 h^2+6 \\
 6&h^5+4 h^3+3 h & 6 h^5+20 h^3+12 h & 21 h^5+60 h^3+30 h \\
 7&h^6+5 h^4+6 h^2+1 & 7 h^6+30 h^4+30 h^2+4 & 28 h^6+105 h^4+90 h^2+10 \\ 
8& h^7+6 h^5+10 h^3+4 h & 8 h^7+42 h^5+60 h^3+20 h & 36 h^7+168 h^5+210 h^3+60 h \\ \hline
 \end{tabular}
\caption{$F_{h, n}^{(r)}(x)$, with $r=1, 2, 3$} \label{tabla1}
\end{table}

\begin{theorem}
The following identities hold:
\begin{enumerate}
\item $F_{h,2}^{(r)}(x)=rh(x)$.
\item $F_{h,n}^{(r)}(x)=F_{h,n}^{(r-1)}(x)+h(x)F_{h,n-1}^{(r)}(x)+F_{h,n-2}^{(r)}(x), \  n\geqslant 2$.
\item $nF_{h,n+1}^{(r)}(x)=r(h(x)F_{h,n}^{(r+1)}(x)+2F_{h,n-1}^{(r+1)}(x)),  \  n\geqslant 1$.
\end{enumerate}
\end{theorem}
\begin{proof}
\begin{enumerate}
\item Taking  $j=1$ in (\ref{hum}), we obtain
\begin{align*}
F_{h, 2}^{(r)}(x)=\binom{r}{1}\binom{1}{0}h(x)=rh(x).
\end{align*}
\item This identity  is obtained from observing that
\begin{align*}
\sum_{j=0}^{\infty}F_{h,j+1}^{(r)}(x)t^j=(h(x)t+t^2)\sum_{j=0}^{\infty}F_{h,j+1}^{(r)}(x)t^j + \sum_{j=0}^{\infty}F_{h,j+1}^{(r-1)}(x)t^j.
\end{align*}
\item Taking the first derivative of  $g_h^{(r)}(t)=(1-h(x)t-t^2)^{-r}$, we obtain
\begin{align*}
(g_h^{(r)}(t))'&=\sum_{j=1}^{\infty}F_{h,j+1}^{(r)}(x)jt^{j-1}=r\left(\frac{1}{1-h(x)t-t^2}\right)^{r-1}\left(\frac{h(x)+2t}{(1-h(x)t-t^2)^2}\right)\\
&=r(h(x)+2t)g_h^{(r+1)}(t)
\end{align*}
Therefore the identity is clear.
\end{enumerate}
\end{proof}
In the next theorem we show that the convolved $h(x)$-Fibonacci polynomials can be expressed in terms of $h(x)$-Fibonacci and $h(x)$-Lucas polynomials. This theorem generalizes Theorem 4 of \cite{com3} and Theorem 4 of \cite{RAM}.
\begin{theorem} Let $j \geqslant 0$ and $r\geqslant 1$. We have
\begin{align}
F_{h,j+1}^{(r)}(x)=&\sum_{\substack{l=0\\ r+l\cong0\mod 2}}^{r-1}\binom{r+l-1}{l}\binom{r-l+j-1}{j}\frac{1}{(h(x)^2+4)^{(r+l)/2}}L_{h,r+j-l}^{(r)}(x) \label{fibolucas} \\
&+ \sum_{\substack{l=0\\ r+l\cong 1 \mod 2}}^{r-1}\binom{r+l-1}{l}\binom{r-l+j-1}{j}\frac{1}{(h(x)^2+4)^{(r+l-1)/2}}F_{h,r+j-l}^{(r)}(x) \notag
\end{align}
\end{theorem}
\begin{proof}
Given $\alpha, \beta\in \mathbb{C}$, such that $\alpha\beta\neq0$ and $\alpha\neq\beta$. Then we have the following partial fraction decomposition:
\begin{align*}
(1-\alpha z)^{-r}(1-\beta z)^{-r}=& \sum_{l=0}^{r-1}\binom{-r}{l}\frac{\alpha^r\beta^l}{(\alpha-\beta)^{r+l}}(1-\alpha z)^{l-r}\\
& +  \sum_{l=0}^{r-1}\binom{-r}{l}\frac{\beta^r\alpha^l}{(\beta-\alpha)^{r+l}}(1-\beta z)^{l-r},
\end{align*}
where $\binom{t}{0}=1$ and $\binom{t}{l}=\frac{t(t-1)\cdots (t-l+1)}{l!}$ with $t\in \mathbb{R}$. Using the Taylor expansion
$$(1-z)^t=\sum_{j=0}^{\infty}(-1)^j\binom{t}{j}z^j$$
Then $(1-\alpha z)^{-r}(1-\beta z)^{-r}=\sum_{j=0}^{\infty}\gamma(j)z^j$, where
\begin{align*}
\gamma(j)=&\sum_{l=0}^{r-1}\binom{-r}{l}\frac{\alpha^r\beta^l}{(\alpha-\beta)^{r+l}}(-1)^j\binom{l-r}{j}\alpha^{j}\\
&+\sum_{l=0}^{r-1}\binom{-r}{l}\frac{\beta^r\alpha^l}{(\beta-\alpha)^{r+l}}(-1)^j\binom{l-r}{j}\beta^{j}
\end{align*}
Note that $1-h(x)z-z^2=(1-r_1(x)z)(1-r_2(x)z)$. On substituting these values of $\alpha=r_1(x)$ and $\beta=r_2(x)$ and using the identities  (\ref{iden}), we obtain
\begin{align*}
F_{h,j+1}^{(r)}(x)=&\sum_{l=0}^{r-1}\binom{-r}{l}\frac{(-1)^l\alpha^{r-l}}{(h(x)^2+4)^{(r+l)/2}}(-1)^j\binom{l-r}{j}\alpha^{j}\\
&+\sum_{l=0}^{r-1}\binom{-r}{l}\frac{(-1)^l\beta^{r-l}}{(h(x)^2-4)^{(r+l)/2}}(-1)^j\binom{l-r}{j}\beta^{j}\\
&=\sum_{l=0}^{r-1}(-1)^l\binom{-r}{l}(-1)^j\binom{l-r}{j}\frac{1}{(h(x)^2+4)^{(r+l)/2}}\left(\alpha^{r+j-l}+(-1)^{r+l}\beta^{r+j-l}\right)
\end{align*}
Since that $(-1)^l\binom{-r}{l}=\binom{r+l-1}{l}$ and $(-1)^j\binom{l-r}{j}=\binom{r-l+j-1}{j}$, then
\begin{align*}
F_{h,j+1}^{(r)}(x)=\binom{r+l-1}{l}\binom{r-l+j-1}{j}\frac{1}{(h(x)^2+4)^{(r+l)/2}}\left(\alpha^{r+j-l}+(-1)^{r+l}\beta^{r+j-l}\right)
\end{align*}
From the above equality and Binet formula, we obtain the equation (\ref{fibolucas}).
\end{proof}
\section{Hessenberg Matrices and Convolved $h(x)$-Fibonacci Polynomials}

An upper Hessenberg matrix, $A_n$, is an $n\times n$  matrix, where $a_{i,j} = 0$ whenever
$i > j+1$ and $a_{j+1,j} \neq 0$ for some $j$. That is, all entries bellow the superdiagonal
are 0 but the matrix is not upper triangular:
\begin{align}
A_n=\begin{pmatrix}
a_{1,1}  & a_{1,2}  & a_{1,3}  &  \cdots  & a_{1,n-1}  & a_{1,n}  \\
a_{2,1}  & a_{2,2}  & a_{2,3}  &  \cdots  & a_{2,n-1}  & a_{2,n}  \\
0  & a_{3,2}  & a_{3,3}  &  \cdots  & a_{3,n-1}  & a_{3,n}  \\
\vdots  & \vdots  & \vdots &  \cdots  & \vdots & \vdots  \\
0  & 0  & 0  &  \cdots  & a_{n-1,n-1}  & a_{n-1,n}  \\
0  & 0  & 0  &  \cdots  & a_{n,n-1}  & a_{n,n}
\end{pmatrix}
\end{align}
We consider a type of upper Hessenberg matrix whose determinants are $h(x)$-Fibonacci numbers. Some results about Fibonacci numbers and Hessenberg can be found in \cite{HES}. The following known result about upper Hessenberg matrices will be used.
\begin{theorem}\label{thes}
Let  $a_1$, $p_{i,j}  (i\leqslant j)$ be arbitrary elements of a commutative ring $R$, and let the sequence $a_1, a_2,\dots$ be defined by:
\begin{align*}
a_{n+1}=\sum_{i=1}^{n}p_{i,n}a_i, \ \ (n=1, 2, \dots).
\end{align*}
If
\begin{align*}
A_n=\begin{pmatrix}
p_{1,1}  & p_{1,2}  & p_{1,3}  &  \cdots  & p_{1,n-1}  & p_{1,n}  \\
-1  & p_{2,2}  & p_{2,3}  &  \cdots  & p_{2,n-1}  & p_{2,n}  \\
0  & -1  & p_{3,3}  &  \cdots  & p_{3,n-1}  & p_{3,n}  \\
\vdots  & \vdots  & \vdots &  \cdots  & \vdots & \vdots  \\
0  & 0  & 0  &  \cdots  & a_{n-1,n-1}  & a_{n-1,n}  \\
0  & 0  & 0  &  \cdots  & -1  & a_{n,n}
\end{pmatrix}
\end{align*}
then
\begin{align}
a_{n+1}=a_1\det A_n.
\end{align}
\end{theorem}
In particular, if
\begin{align}
F_n^{(h)}=\begin{pmatrix}
h(x)  & 1  & 0  &  \cdots  & 0  & 0  \\
-1  & h(x)  & 1  &  \cdots  & 0  & 0  \\
0  & -1  & h(x)  &  \cdots  & 0  & 0  \\
\vdots  & \vdots  & \vdots &  \cdots  & \vdots & \vdots  \\
0  & 0  & 0  &  \cdots  & h(x) & 1  \\
0  & 0  & 0  &  \cdots  & -1  & h(x)
\end{pmatrix}
\end{align}
then  from Theorem \ref{thes} we have that
\begin{align}
\det F_n^{(h)}=F_{h,n+1}(x),  \ \ \ (n=1, 2, \dots).
\end{align}
It is clear that the principal minor $M^{(h)}(i)$ of $F_n^{(h)}$ is equal to $F_{h,i}(x)F_{h,n-i+1}(x)$. It follows that the principal minor  $M^{(h)}(i_1,i_2,\dots,i_l)$ of the matrix $F_n^{(h)}$ is obtained by deleting rows and columns with indices $1\leqslant i_1 < i_2 < \cdots < i_l\leqslant n$:
\begin{align}
M^{(h)}(i_1,i_2,\dots,i_l)=F_{h,i_1}(x)F_{h,i_2-i_1}(x)\cdots F_{h,i_l-i_{l-1}}(x)F_{h,n-i_l+1}(x).
\end{align}
Then we have the following theorem.
\begin{theorem}
Let $S_{n-l}^{(h)}, (l=0, 1, 2, \dots, n-1)$ be the sum of all principal minors of $F_n^{(h)}$ or order $n-l$. Then
\begin{align}
S_{n-l}^{(h)}=\sum_{j_1+j_2+\cdots + j_{l+1}=n-l}F_{h,j_1+1}(x)F_{h,j_2+1}(x)\cdots F_{h,j_{l+1}+1}(x)=F_{h,n-l+1}^{(l+1)}(x).
\end{align}
\end{theorem}

Since the coefficients  of the characteristic polynomial of a matrix are, up to the sign, sums of principal minors of the matrix, then we have the following. 

\begin{corollary}\label{corofibo}
The convolved $h(x)$-Fibonacci polynomials $F_{h,n-l+1}^{(l+1)}(x)$ is equal, up to the sign, to the coefficient of $t^l$ in the characteristic polynomial $p_{n}(t)$ of $F_n^{(h)}$.
\end{corollary}

\begin{corollary}
The following identity holds:
\begin{align*}
F_{h,n-l+1}^{(l+1)}(x)=\sum_{i=0}^{\lfloor (n-l)/2 \rfloor}\binom{n-i}{i}\binom{n-2i}{l}h(x)^{n-2i-l}.
\end{align*}
\end{corollary}
\begin{proof}
The characteristic matrix of $F_{n}^{(h)}$ has the form
\begin{align}
\begin{pmatrix}
t-h(x)  & 1  & 0  &  \cdots  & 0  & 0  \\
-1  & t-h(x)  & 1  &  \cdots  & 0  & 0  \\
0  & -1  & t-h(x)  &  \cdots  & 0  & 0  \\
\vdots  & \vdots  & \vdots &  \cdots  & \vdots & \vdots  \\
0  & 0  & 0  &  \cdots  & t-h(x) & 1  \\
0  & 0  & 0  &  \cdots  & -1  & t-h(x)
\end{pmatrix}
\end{align}
Then $p_n(t)=F_{n+1}(t-h(x))$, where $F_{n+1}(t)$ is a Fibonacci polynomial. Then from Corollary \ref{corofibo}   and the following identity  for Fibonacci polynomial \cite{falcon5}:
   \begin{align*}
   F_{n+1}(x)=\sum_{i=0}^{\lfloor n/2 \rfloor}\binom{n-i}{i}x^{n-2i},
   \end{align*}
 we obtain that
 \begin{align*}
   F_{n+1}(t-h(x))=\sum_{i=0}^{\lfloor n/2 \rfloor}\binom{n-i}{i}\sum_{l=0}^{n-2i}\binom{n-2i}{l}(-1)^{n-l}h(x)^{n-2i-l}t^l.
   \end{align*}
 Therefore the corollary is obtained.
\end{proof}

\section{Acknowledgments}
The author would like to thank the anonymous referees for their helpful comments.  The author was partially supported by Universidad Sergio Arboleda under Grant no. USA-II-2012-14.

\end{document}